\documentclass{amsart}
\usepackage[utf8]{inputenc}
\usepackage{amssymb}
\usepackage{hyperref}
\usepackage{a4wide}

\numberwithin{equation}{section}
\newtheorem{theorem}{Theorem}[section]

\newtheorem{maintheorem}{Theorem}

\newtheorem{lemma}[theorem]{Lemma}
\theoremstyle{proposition}
\newtheorem{proposition}[theorem]{Proposition}
\theoremstyle{corollary}

\theoremstyle{definition}
\newtheorem{definition}[theorem]{Definition}
\newtheorem{remark}[theorem]{Remark}
\newtheorem{remarks}[theorem]{Remarks}
\newtheorem{example}[theorem]{Example}
\newtheorem{examples}[theorem]{Examples}

 \DeclareMathOperator\Ad{{\rm Ad}}
 \DeclareMathOperator\Aut{{\rm Aut}}

 \DeclareMathOperator\rk{{\rm rk}}

\newcommand\Liegrp{\sf }

\newcommand\LA{\Liegrp A}
\newcommand\LB{\Liegrp B}
\newcommand\LC{\Liegrp C}
\newcommand\LD{\Liegrp D}

\newcommand\LG{\Liegrp G}

\newcommand\SO{\Liegrp SO}
 \newcommand\Sp{\Liegrp Sp}
\newcommand\Spin{\Liegrp Spin}
\newcommand\SU{\Liegrp SU}
\newcommand\SUxU[2]{\Liegrp S(U(#1)\times\Liegrp U(#2\Liegrp ))} \newcommand\U{\Liegrp U}

\newcommand\eS{{\rm S}}
\newcommand\g{\mathfrak g}

\newcommand\Gr{\rm Gr}
\newcommand\h{\mathfrak h}

\newcommand\id{\rm id}

\newcommand\p{\mathfrak p}

\newcommand\R{\mathbb R}

\newcommand\str{\rule[-.48em]{0em}{1.7em}}

\renewcommand\k{{\mathfrak k}}

\begin{document}
\setcounter{tocdepth}{1}

%
%
%

\pagenumbering{arabic}
\pagestyle{plain}

\title{Hyperpolar actions on reducible symmetric spaces}

\author[A. Kollross]{Andreas Kollross }

\address{Universität Stuttgart, Institut für Geometrie und Topologie}

\email{kollross@mathematik.uni-stuttgart.de}

\date{\today}



\begin{abstract}
We study hyperpolar actions on reducible symmetric spaces of the compact type. Our main result is that an indecomposable hyperpolar action on a symmetric space of the compact type is orbit equivalent to a Hermann action or of cohomogeneity one.
\end{abstract}

\maketitle


\section{Introduction and Main Results}
\label{sec:Intro}


An isometric action of a compact Lie group on a Riemannian manifold is called \emph{polar} if there exists an immersed submanifold which meets every orbit such that the orbits intersect the submanifold orthogonally at each of its points. Such a submanifold is called a \emph{section} of the Lie group action. If there is a section which is flat in its induced Riemannian metric, then the action is called \emph{hyperpolar}.

Polar and hyperpolar actions have been studied by Conlon~\cite{conlon}, Szenthe~\cite{szenthe}, Palais and Terng~\cite{pt}. The problem of classifying hyperpolar actions on compact symmetric spaces was posed by Heintze, Palais, Terng and Thorbergsson in~\cite{hptt}. Natural examples for hyperpolar actions are given by the isotropy actions and isotropy representations of symmetric spaces. Moreover, actions of cohomogeneity one, i.e.\ actions whose orbits of maximal dimension are hypersurfaces, and the so-called Hermann~\cite{hermann} actions, see Section~\ref{sec:Hermann} for a definition, are well-known examples of hyperpolar actions.
The main result of this article is the following.

\begin{maintheorem}\label{main:Hermann}
An indecomposable hyperpolar action of cohomogeneity greater than one on a Riemannian symmetric space of compact type is orbit equivalent to a Hermann action.
\end{maintheorem}

In the special case where the symmetric space is irreducible, this follows from the classification of hyperpolar actions obtained by the author in~\cite{hyperpolar}. Note that the indecomposability of a hyperpolar action does not imply that the space acted on is irreducible, as was pointed out in~\cite{hl}. See Section~\ref{sec:IsoActCpt} on how to construct indecomposable actions with arbitrarily many irreducible factors, see Section~\ref{sec:Examples} for further examples. Theorem~\ref{main:Hermann} implies the following splitting theorem for hyperpolar actions.

\begin{maintheorem}\label{main:Splitting}
Assume the compact connected Lie group~$H$ acts hyperpolarly on the Riemannian symmetric space~$M$ of the compact type. Then, possibly after replacing~$H$ by a larger orbit equivalent group, there are splittings $H= H_1 \times \dots \times H_n$ and $M = M_1 \times \dots \times M_n$ such that the following holds. The $H$-action on~$M$ is orbit equivalent to the product of all $H_i$-actions on the~$M_i$. For each $i \in \{1,\dots,n\}$ the action of~$H_i$ on~$M_i$ is one of the following:
\begin{enumerate}
  \item transitive, in which case $M_i$ is irreducible;
  \item indecomposable and of cohomogeneity one;
  \item an indecomposable Hermann action.
\end{enumerate}
\end{maintheorem}

The proof of Theorems~\ref{main:Hermann} and~\ref{main:Splitting} is based on a partial classification of hyperpolar actions on products of two irreducible compact symmetric spaces. See \cite{kl}, \cite{kraly}, \cite{lytchak} for similar recent results on polar actions with non-flat sections. See~\cite{dr} for a survey on polar and hyperpolar actions.

This article is organized as follows. After presenting some preliminary notions and facts, a basic construction for actions on compact symmetric spaces is introduced in Section~\ref{sec:IsoActCpt}. In Section~\ref{sec:HypProd}, some criteria for hyperpolarity, in particular for actions on products of symmetric spaces, are given. In Section~\ref{sec:Hermann}, Hermann actions are defined and indecomposable Hermann actions are classified. In Section~\ref{sec:Trans}, the results of Oni\v{s}\v{c}ik~\cite{oniscik} on transitive actions are reviewed; they are needed in Section~\ref{sec:TwoFactors}, where our main result is proved in the special case of two irreducible factors. In Section~\ref{sec:ManyFactors}, the classification of indecomposable hyperpolar actions from the previous section is generalized to spaces with arbitrarily many irreducible factors, thus proving Theorem~\ref{main:Hermann}. In Section~\ref{sec:Examples}, examples of indecomposable cohomogeneity one actions, which are not orbit equivalent to Hermann actions, are given. Some open questions are stated in the last section.


\section{Preliminaries}
\label{sec:Prelim}


Let $G$ be a connected Lie group and let $K \subseteq G$ be a closed subgroup. The pair $(G,K)$ is called a \emph{symmetric pair} if there exists an involutive automorphism~$\sigma$ of~$G$ such that $(G^\sigma)^0 \subseteq K \subseteq G^\sigma$, where we denote by $G^\sigma$ the set of fixed points of~$\sigma$ and by~$(G^\sigma)^0$ its identity component. If $(G,K)$ is a symmetric pair, we say that $K$ is a \emph{symmetric subgroup} of~$G$. Let $\tilde G$ be the universal cover of~$G$. We say that $K$ is a \emph{locally symmetric subgroup} of~$G$ if there exists a symmetric subgroup $\tilde K \subset \tilde G$ such that $K^0 = p(\tilde K)^0$, where $p \colon \tilde G \to G$ is the covering map.

We use the term \emph{subaction} to refer to the restriction of a Lie group action $G \times M \to M$ to $H \times M$, where $H \subseteq G$  is a closed subgroup. If $M_1$ and $M_2$ are any sets, we always denote by $\pi_i \colon M_1 \times M_2 \to M_i$ the natural projection $(m_1,m_2) \mapsto m_i$ for $i=1,2$. If an action of the group~$G_1$ on the set~$M_1$ and an action of the group~$G_2$ on the set~$M_2$ is given, we define the \emph{product action} of $G_1 \times G_2$ on $M_1 \times M_2$ by $(g_1,g_2) \cdot (m_1,m_2) := (g_1 \cdot m_1, g_2 \cdot m_2)$. For any group $H$, we define the \emph{diagonal subgroup}~${\rm \Delta} H$ of~$H \times H$ by
$
{\rm \Delta} H := \left\{ (h,h) \,\middle\vert\, h \in H \right\}.
$
More generally, if $H_1$ and $H_2$ are two locally isomorphic Lie groups and $H$ is a Lie group such that local isomorphisms $\phi_1 \colon H \to H_1$ and $\phi_2 \colon H \to H_2$ exist, then we say that $ \left\{ (\phi_1(h),\phi_2(h)) \,\middle\vert\, h \in H \right\} $ is a \emph{diagonal subgroup} of $H_1 \times H_2$. If $G$ is a compact Lie group endowed with a biinvariant Riemannian metric, then the connected component of the isometry group is covered by $G \times G$, where the action of an element $(h_1,h_2) \in G \times G$ on~$G$ is given by
\begin{equation}\label{eq:GAct}
(h_1,h_2) \cdot g := h_1 \, g \, h_2^{-1}.
\end{equation}
Henceforth, if $H$ is a connected compact Lie group acting isometrically on a compact Lie group~$G$ with biinvariant metric, we will always assume that $H$ is given by a closed subgroup of~$G \times G$ and, conversely, if $H \subseteq G \times G$ is a closed subgroup it will be understood that the $H$-action on $G$ is given by (\ref{eq:GAct}). We say that isometric actions of the same Lie group on two Riemannian manifolds~$M$ and~$N$ are~\emph{conjugate} if there is an equivariant isometry between $M$ and~$N$. We say that two isometric actions on Riemannian manifolds $M$ and~$N$ are \emph{locally conjugate} if there exists a Riemannian manifold~$V$ with an isometric Lie group action and surjective equivariant local isometries $V \to M$ and $V \to N$. We say that two isometric actions on two Riemannian manifolds $M$ and $N$ are \emph{orbit equivalent} if there is an isometry $M \to N$ which maps each connected component of an orbit in~$M$ to a connected component of an orbit in~$N$. We will denote the isometry group of a Riemannian manifold~$M$ by~$I(M)$ and by $I(M)^0$ its connected component.

\begin{definition}\label{def:ProjIntAct}
Let $M_1$ and $M_2$ be Riemannian manifolds and let $I(M_1)$ and $I(M_2)$ be their isometry groups. Let $H \subseteq I(M_1) \times I(M_2)$ be a closed subgroup.
\begin{enumerate}
\item We define an isometric action of $H$ on one of the factors~$M_i$ by
\[
(h_1,h_2) \cdot p_i := h_i \cdot p_i \quad\hbox{for $(h_1,h_2) \in H$, $p_i \in M_i$.}
\]
We call this the \emph{projection action of $H$ on $M_i$} or the \emph{$H$-action on~$M_i$}.

\item Let $o_1 \in M_1$. Then we define
\[
H_{o_1} := \{ h \in H \mid h \cdot o_1 = o_1 \}.
\]
The subgroup $H_{o_1}$ of~$H$ is called the \emph{partial isotropy group of the $H$-action on~$M$ at~$o_1$}.

\item Let $o_1 \in M_1$ and let $H_{o_1}$ be the partial isotropy group as defined in~(ii). The $H$-action on~$M_2$, restricted to~$H_{o_1} \times M_2$, is called the \emph{intersection action of $H_{o_1}$ on~$M_2$}.

\end{enumerate}
\end{definition}

The orbits of the projection action are exactly the projections of the $H$-orbits on~$M$ to~$M_i$, i.e.\ we have $H \cdot o_i = \pi_i( H \cdot (o_1,o_2) )$ for $i=1,2$, $o_i \in M_i$. The orbits of the intersection action of $H_{o_1}$ on~$M_2$ are given by the intersections of the $H$-orbits on~$M$ with $\{o_1\} \times M_2$, more precisely, we have $\{o_1\} \times (H_{o_1} \cdot o_2) = (H \cdot (o_1,o_2)) \cap (\{o_1\} \times M_2)$ for $o_1 \in M_1, o_2 \in M_2$.

In the special case where the $H$-action on~$M_1$ is transitive, all the partial isotropy groups $H_{o_1}, o_1 \in M_1,$ are conjugate in~$H$ and hence all intersection actions of $H_{o_1}$, $o_1 \in M_1$ on~$M_2$ are conjugate.

Note that for general Riemannian manifolds $M_1,M_2$ the isometry group $I(M_1 \times M_2)$ of the Riemannian product contains $I(M_1) \times I(M_2)$ as a subgroup, but is in general a larger group. For instance, if $M_1$ and $M_2$ are Euclidean spaces, then we have $I(\R^{n+m}) \neq I(\R^n) \times I(\R^m)$ if $m,n \ge 1$. Thus, projection actions and intersection actions are not defined for general products of Riemannian manifolds. However, to study actions of connected Lie groups on reducible symmetric spaces of the compact type, they are useful: Let $M_1$ and $M_2$ be connected Riemannian symmetric spaces whose universal covers are without Euclidean factors. Then the connected component of the isometry group of the Riemannian product $M = M_1 \times M_2$ is given by $I(M)^0 = I(M_1)^0 \times I(M_2)^0$. Hence any connected subgroup~$H$ of $I(M)$ is contained in $I(M_1) \times I(M_2)$ and the projection actions of~$H$ on~$M_i$ are well defined for $i=1,2$.
Since the (hyper-)polarity of an action depends only on the connected component of the group acting, see Remark~\ref{rem:hypliealg} below, we may restrict ourselves to actions of connected Lie groups in the following.

\begin{definition}\label{def:decomp}
We say that an isometric action of a Lie group~$H$ on a product of two Riemannian manifolds $M_1 \times M_2$ \emph{decomposes} if there exist Lie groups $H_1$ and $H_2$ such that $H_i$ acts isometrically on~$M_i$ for $i=1,2$ and the $H$-action is orbit equivalent to the product action of $H_1 \times H_2$ on $M_1 \times M_2$.
We say that an isometric action of a Lie group on a Riemannian manifold $M$ is \emph{decomposable} if there exist Riemannian manifolds $M_1$ and $M_2$ such that $M = M_1 \times M_2$ as a Riemannian product and the action of~$H$ on the product $M_1 \times M_2$ decomposes. Otherwise, we say the action is \emph{indecomposable}.
\end{definition}

The following criterion for the hyperpolarity of an action is well known.

\begin{proposition}\label{prop:PolCrit}
Let $G$ be a connected semisimple compact Lie group, let $\sigma$ be an involutive automorphism of~$G$ and let $K \subset G$ be a closed subgroup such that $(G^\sigma)^0 \subseteq K \subseteq G^\sigma$. Let $\g = {\mathfrak k} \oplus \p$ be the decomposition into eigenspaces of $d\sigma_e$. Let $G/K$ be endowed with the Riemannian metric induced by an $\Ad(G)$-invariant inner product on~$\g$. Let $H \subseteq G$ be a closed subgroup, acting on~$G/$K with cohomogeneity~$d$. Then the $H$-action on~$G/K$ is hyperpolar if and only if $N_{eK}(H \cdot eK) \subseteq \p$ contains a $d$-dimensional abelian subspace.
\end{proposition}

\begin{proof}
See~\cite[Proposition~4.1]{polar}
\end{proof}

\begin{remarks}\label{rem:hypliealg}
It follows from the proposition that the hyperpolarity of an action can be decided on the Lie algebra level. Therefore we may restrict ourselves to consider one representative of each local isometry class of symmetric spaces. For the same reason, we may assume that the groups acting hyperpolarly on symmetric spaces are connected.
\end{remarks}

Since sections of polar actions are totally geodesic submanifolds, it follows that we have
\begin{equation}\label{eq:dim}
\dim(H) \ge \dim(M) - \rk(M)
\end{equation}
whenever a compact Lie group~$H$ acts hyperpolarly on a Riemannian symmetric space~$M$.


\section{Expanding and Reducing Factors}\label{sec:IsoActCpt}


In this section, we introduce a construction which relates certain isometric actions on compact symmetric spaces. The construction described below can be stated in a sort of shorthand notation as follows:
\[
H \curvearrowright M_1 \times  G_2/K_2 \qquad \rightsquigarrow \qquad H \times K_2 \curvearrowright M_1 \times G_2.
\]
i.e.~an action of $H$ on~$M_1 \times  G_2/K_2$ is transformed into an action of $H \times K_2$ on~$M_1 \times G_2$. This construction is motivated by~\cite[Proposition~2.11]{hptt}, where it was observed that, in the special case where $M_1$ is trivial, the $H$-action on $G_2/K_2$ is hyperpolar if and only if the $H \times K_2$-action on~$G_2$ is hyperpolar, cf.~Theorem~\ref{thm:str} below.

There are two types of irreducible symmetric spaces of the compact type, see~\cite{helgason}: The spaces of \emph{Type~I} and \emph{Type~II}. The spaces of Type~I are those with simple isometry group; they are exactly the Riemannian symmetric spaces which have a homogeneous presentation~$G/K$ where $G$ is a simple compact Lie group and $(G,K)$ is a symmetric pair. The spaces of Type~II are the simple compact Lie groups~$L$ endowed with a biinvariant Riemannian metric; their isometry group is finitely covered by $L \times L$ and so they have the homogeneous presentation $(L \times L) / {\rm \Delta} L$.

Let $M$ be a Riemannian symmetric space such that $M = M_1 \times M_2$ where the factors~$M_1$ and~$M_2$ are, not necessarily irreducible, symmetric spaces of the compact type. Let $G = G_1 \times G_2$, where $G_i$ is the universal cover of the connected component of the isometry group of~$M_i$, cf.~\cite[Theorem~8.3.9]{wolf}. We may assume that the Riemannian metric on~$M$ is induced by an $\Ad(G)$-invariant inner product~$\mu$ on~$\g$. Let $H$ be a connected compact Lie group acting isometrically and almost effectively on~$M$. We may assume that $H$ is a connected closed subgroup of~$G$, replacing $H$ by a finite cover, if necessary. Let $o \in M$ be an arbitrary reference point and let $K = G_o$ be its isotropy subgroup with respect to the $G$-action. Define $K_i = K \cap G_i$ for $i = 1, 2$. Let $\g = \k \oplus \p$ as usual. Define $\p_i = \p \cap \g_i$.

We will now define a Lie group action closely related to the $H$-action on~$M$. Roughly speaking, we replace $H$ by~$H \times K_2$ and $M$ by a symmetric space~$\bar M$ such that $\bar M / K_2 \cong M$.
Let
\begin{equation}\label{eq:barM}
\bar M = M_1 \times  G_2.
\end{equation}
Let $G_2$ be endowed with the biinvariant Riemannian metric induced by $\mu|{\g_2 \times \g_2}$. Let $M_1$ be endowed with the invariant Riemannian metric induced by $\mu|{\g_1 \times \g_1}$.
Define the natural projection $\pi \colon \bar M \to M$ by
$
\pi(m_1, g_2) = (m_1, g_2K_2).
$
With our choice of Riemannian metrics, $\pi$ is a Riemannian submersion.
Let
\begin{equation}\label{eq:barH}
\bar H = H \times K_2.
\end{equation}
Define an action of~$\bar H$ on~$\bar M$ as follows:
Let $\bar h = (h_1,h_2,k) \in \bar H$, $m_1 \in M_1$, $g_2 \in G_2$ and let $\bar m = (m_1,g_2)$. Define
\begin{equation}\label{eq:baract}
\bar h \cdot \bar m := (h_1,h_2,k_2) \cdot (m_1, g_2) := (h_1 \cdot m_1, \; h_2 \, g_2 \, k_2^{-1}).
\end{equation}
Note that this action can also be viewed as a subaction of the action of $\bar G = G_1 \times G_2 \times K_2$ on $\bar M$ given by $(g_1,g_2,k_2) \cdot (x, y) := (g_1 \cdot x,\; g_2 \, y \, k_2^{-1})$. Formally, we also consider the $\bar G$-action on~$M$ given by $(g,k_2) \cdot m = g \cdot m$, i.e.\ the action given by the
$G$-action on $M$, where the $K_2$-factor is ignored. With respect to these actions, the map $\pi \colon \bar M \to M$ is equivariant.

\begin{definition}\label{def:Expand}
We say that the $\bar H$-action on~$\bar M$ as introduced above is obtained from the $H$-action on~$M$ by \emph{expanding the factor~$M_2$}. Conversely, if there is a Riemannian symmetric space isometric to~(\ref{eq:barM}), together with an action of a group~$\bar H$ as in~(\ref{eq:barH}), where $\bar H$ acts on~$\bar M$ as described in~(\ref{eq:baract}), then we say that the $H$-action on~$M$ is obtained from the $\bar H$-action on~$\bar M$ by \emph{reducing the factor~$G_2$}. If there exists an $H$-action on a symmetric space~$M$ which is obtained by reducing a factor~$G_2$ in an $\bar H$-action on~$\bar M$, then we say the $\bar H$-action is \emph{reducible}, otherwise we say it is \emph{irreducible}.
\end{definition}

Note that the process of expanding a factor can be repeated arbitrarily many times by expanding any factor of the space~$\bar M$ obtained in the previous step. Since $\dim \bar M = \dim M + \dim K_2$ and $\dim K_2 > 0$, the dimension of the spaces thus generated strictly increases with each step. Further note that if the factor~$M_2$ is a symmetric space of Type~I, then $G_2$ will be of Type~II; if the factor~$M_2$ is a symmetric space of Type~II, then $G_2$ will be the Riemannian product $M_2 \times M_2$; thus we may produce actions on reducible symmetric spaces with arbitrarily many irreducible factors by repeating the process. On the other hand, given any symmetric space $\bar M$ of the compact type, together with a reducible $\bar H$-action, successively reducing factors will lead to an irreducible action in finitely many steps.

Let us remark that the equivariant submersion $\pi \colon \bar M \to M$ induces a bijection between the orbit space of the $H$-action on~$M$ and the orbit space of the $\bar H$-action on~$\bar M$.

\begin{definition}\label{def:grouplift}
Let $M$ be a simply connected Riemannian symmetric space of compact type. Let $M = M_{\rm I} \times M_{\rm II}$ be such that $M_{\rm I}$ is a product of symmetric spaces of Type~I and $M_{\rm II}$ is a product of symmetric spaces of Type~II. Then the action obtained by expanding the factor $M_{\rm I}$ is called the \emph{group lift} of the original action.
\end{definition}

Clearly, if $M_{\rm I}$ is trivial, then the $H$-action on~$M$ is its own group lift. It follows from Theorem~\ref{thm:str} below that an isometric action on a Riemannian symmetric space of the compact type is hyperpolar if and only if its group lift is. Therefore, to study hyperpolar actions on these spaces, we may restrict ourselves to actions on symmetric spaces of Type~II.

\begin{lemma}\label{lm:orbeqex}
Let $M = M_1 \times M_2$, where $M_1$ and $M_2$ are simply connected Riemannian symmetric space of the compact type. Assume $H$ and $L$ are compact connected Lie groups acting isometrically on~$M$. Consider the actions of $\bar H$ and $\bar L$ on~$\bar M$ which are obtained from the $H$-action and the $L$-action by expanding the factor~$M_2$. Then the $H$-action on~$M$ is orbit equivalent to the $L$-action on~$M$ if and only if the $\bar H$-action on~$\bar M$ is orbit equivalent to the $\bar L$-action on~$\bar M$.
\end{lemma}

\begin{proof}
It follows from the above mentioned fact that $\pi$ induces a bijection between orbit spaces that the orbits of the $H$-action on~$M$ agree with the orbits of the $L$-action on~$M$ if and only if the orbits of the $\bar H$-action on~$\bar M$ agree with the orbits of the $\bar L$-action on~$\bar M$.
\end{proof}

\begin{theorem}\label{thm:str}
Let $M_1$ and $M_2$ be simply connected Riemannian symmetric spaces of the compact type and let $H$ be a closed connected subgroup of the universal cover~$G$ of the connected component of the isometry group of~$M = M_1 \times M_2$. Assume the $\bar H$-action on~$\bar M$ is obtained from the $H$-action on~$M$ by expanding the factor~$M_2$. Then the following hold.
\begin{enumerate}

  \item The $\bar H$-action on~$\bar M$ is hyperpolar if and only if the $H$-action on~$M$ is.

  \item The $\bar H$-action on~$\bar M$ is decomposable if and only if the $H$-action on~$M$ is.

  \item If $\pi(\bar o) = o$, then the isotropy groups of~$\bar o$ and~$o$ are isomorphic, i.e.\ $\bar H_{\bar o} \cong  H_o$. Moreover, their slice representations are equivalent.

\end{enumerate}
\end{theorem}

It follows from the statement on the isotropy groups in part~(iii) of Theorem~\ref{thm:str} that the codimension of the orbit $\bar H \cdot \bar o$ in~$\bar M$ equals the codimension of the orbit $H \cdot \pi(\bar o)$ in~$M$. In particular, the $H$-action on~$M$ and the $\bar H$-action on~$\bar M$
are of the same cohomogeneity.

\begin{proof}
We show part~(iii) first. Let us compute the two isotropy groups. Let $\bar o = (m_1, g_2)$ and let
$o = (m_1, m_2)$.
We have
\[
H_o = \left \{ (h_1, h_2) \in H \mid h_j \in g_j K_j g_j^{-1} \mbox{~for~$j = 1,2$} \right \},
\]
where we assume $m_j=g_jK_j$ for $j = 1,2$ and
\[
\bar H_{\bar o} = \left \{ (h_1, h_2, g_2^{-1} h_2 g_2) \in \bar H \mid h_j \in g_j K_j g_j^{-1} \mbox{~for~$j = 1, 2$} \right \}.
\]
An isomorphism $\bar H_{\bar o} \to H_o$ is obviously given by the projection on the first two components.

We may assume $\bar o = ( eK_1, e)$ and $o = ( eK_1, eK_2)$, replacing $\bar H$ and $H$ by conjugate subgroups in~$\bar G$ and $G$, if necessary.
The Lie algebra of~$\bar H$ is
\[
\bar \h = \left \{ (X_1,X_2,Y) \mid (X_1,X_2) \in \h, Y \in \k_i \right \}.
\]
The normal space to the $\bar H$-orbit at~$\bar o$ consists of all elements $(Z_1,Z_2)$ where $Z_1 \in \p_1$, $Z_2 \in \g_2$ and such that $\mu(Z_1,X_1) + \mu(Z_2,X_2) - \mu(Z_2,Y) = 0$ for all $(X_1,X_2) \in \h$ and all $Y \in \k_2$. However, since this condition is equivalent to the condition that $\mu(Z_1,X_1) + \mu(Z_2,X_2) = 0$ for all $(X_1,X_2) \in \h$ and $Z_2 \in \p_2$, we actually have $N_{\bar o} (\bar H \cdot \bar o) = N_o(H \cdot o)$,
using the identifications $T_{\bar o}\bar M = \p_1 \times \g_2$ and $T_o\bar M = \p_1 \times \p_2$.

{}From this it follows that the slice representations of $\bar H_{\bar o}$ on $N_{\bar o} (\bar H \cdot \bar o)$ and of~$H_o$ on $N_o(H \cdot o)$ are equivalent, where we use the isomorphism described in the first part of the proof to identify the two isotropy groups $\bar H_{\bar o}$ and~$H_o$. We have shown~(iii).

In the above situation, we may additionally assume that $\bar o$ is a regular point of the $\bar H$-action on~$\bar M$. It then follows that~$o$ is a regular point of the $H$-action on~$M$ by~(iii). Since we have just shown that the normal spaces of the two actions at~$\bar o$ and $o$ agree, (i) follows from Proposition~\ref{prop:PolCrit}.

We will now prove~(ii). For the purposes of this proof, let us assume that $M = M_1 \times \dots \times M_n$ where the factors are irreducible. It suffices to prove (ii) in the case where the $\bar H$-action on~$\bar M$ is obtained from the $H$-action on~$M$ by expanding the irreducible factor~$M_1$. Assume $\{1, \dots ,n\}$ is the disjoint union of two nonempty sets $I$ and $J$, where $1 \in I$. Let $M_I = \prod_{j \in I} M_j$ and $M_J = \prod_{j \in J} M_j$.

First assume the $H$-action on~$M$ is decomposable in such a way that the $H$-action on~$M$ is orbit equivalent to the product of the action of a Lie group~$H_I$ on~$M_I$ and the action of a Lie group~$H_J$ on~$M_J$. Then by Lemma~\ref{lm:orbeqex}, the $\bar H$-action on~$\bar M$ is orbit equivalent to the product of the action of~$\bar H_I$ on~$\bar M_I$ and the action of ~$H_J$ on~$M_J$.

Now assume the $\bar H$-action on~$\bar M$ is decomposable in such a way that the $\bar H$-action on~$\bar M$ is orbit equivalent to the product of the action of a Lie group~$\bar H_I$ on~$\bar M_I$ and the action of a Lie group~$H_J$ on~$M_J$. Then by Lemma~\ref{lm:orbeqex}, the $H$-action on~$M$ is orbit equivalent to the product of the action of~$H_I$ on~$M_I$ and the action of ~$H_J$ on~$M_J$.

This suffices to prove~(ii) in case $M_1$ is a symmetric space of Type~I, since in this case~$G_1$ is a simple compact Lie group and hence an irreducible symmetric space. However, if $M_1$ is of Type~II, then $G_1$ is a product of two isomorphic simple compact Lie groups and thus the following additional case may arise:
Assume the action of $\bar H = H \times G_1$ on~$\bar M$ is decomposable in such a way that the $\bar H$-action on~$\bar M$ is orbit equivalent to the product of the action of a Lie group~$H_+$ on~$M_+ := \prod_{j \in I \setminus \{1\}} M_j \times G_1$ and the action of a Lie group~$H_-$ on~$M_- := \prod_{j \in J} M_j \times G_1$. In this case the $\bar H$-action on $\bar M$ is orbit equivalent to the product of the two projection actions, namely of the $\bar H$-action on~$M_+$ and the $\bar H$-action on~$M_-$. However, for the first action the orbits are of the form $X \times G_1$, where $X \subseteq M_I$ and for the second action the orbits are of the form $Y \times G_1$, where $Y \subseteq M_J$ , respectively. It follows that all orbits of the $H$-action on~$M$ are of the form $X  \times Y \times M_1$. Hence the $H$-action on~$M$ is decomposable.
\end{proof}


\section{Hyperpolar Actions on Products}\label{sec:HypProd}


\begin{lemma}\label{lm:indechp}
Let $M_1$ and $M_2$ be Riemannian symmetric spaces of the compact type and let $H$ be a closed connected subgroup of the universal cover~$G$ of the connected component of the isometry group of~$M = M_1 \times M_2$.
Assume $H$ acts isometrically on~$M$ and such that the action is hyperpolar and indecomposable. Then both the $H$-actions on $M_1$ and $M_2$ are transitive.
\end{lemma}

\begin{proof}
Using Theorem~\ref{thm:str}, it suffices to prove the lemma for the group lift of the $H$-action on~$M$. Thus we may assume $M$ is a Riemannian product of symmetric spaces of Type~II.

Let $\Sigma \subset M$ be a section of this action. It follows from \cite[Theorem~B]{hl} that the action of the Weyl group $W(\Sigma)$ on~$\Sigma$ is indecomposable. In particular, there is an element $g \in M$ such that the slice representation of the $H$-action at~$g$ is irreducible. After conjugating~$H$, if necessary, we may assume that $g$ is the identity element~$e$ of~$M$. Let $V = N_e(H \cdot e)$ be the normal space to the orbit through~$e$. Then $V$ is the representation space of the slice representation of $H_e$. For $i=1,2$, let $\pi_i \colon M \to M_i$ be the natural projection. Since $\pi_i$ is equivariant with respect to the $H$-actions on $M$ and $M_i$, it follows that its differential $D \pi_i(e)$ restricted to~$V$ becomes an intertwining map with respect to the $H_e$-representation on~$V$. Hence the kernel $V_i$ is an invariant subspace. Since $V$ is an irreducible representation of $H_e$, it follows that $V = V_i$ if we assume $V_i$ is non-trivial, which then implies that $V_{3-i}$ is trivial. Then the $H$-action on $M$ is decomposable, since it is orbit equivalent to the product of two actions one of which is transitive, a contradiction.  This proves that $V_1 = V_2 = 0$ and both projection actions of $H$ on $M_1$ and on~$M_2$ are transitive.
\end{proof}

\begin{lemma}\label{lm:maxsubg}
Assume the compact Lie group~$L$ acts isometrically on a symmetric space~$M$ of compact type. Let $H \subseteq L$ be a closed subgroup such that the $L$-action on~$M$ restricted to~$H$ is hyperpolar and indecomposable. Then the $H$-action on~$M$ is orbit equivalent to the $L$-action or the $L$-action is transitive.
\end{lemma}

\begin{proof}
This follows with the same argument as in the proof of~\cite[Corollary~3.14]{hptt}, where instead of the indecomposability of the $H$-action it is assumed that the affine Coxeter group associated with the $H$-action is irreducible. It was later proved in \cite[Theorem~B]{hl} that this assumption is equivalent to the indecomposability of the action.
\end{proof}

We will now prove the following characterization of indecomposable hyperpolar actions in terms of projection actions and intersection actions.

\begin{proposition}\label{prop:NonSplt}
Let $M = M_1 \times M_2$ where $M_1$ and $M_2$ are Riemannian symmetric spaces of the compact type. Let $o=(o_1,o_2) \in M_1 \times M_2$. Assume the compact Lie group~$H$ acts isometrically on~$M$ and such that the action is indecomposable. Then the action is hyperpolar if and only if all of the following hold.
\begin{enumerate}
  \item The $H$-action on $M_1$ is transitive.
  \item The $H$-action on $M_2$ is transitive.
  \item The $H_{o_1}$-action on $M_2$ is hyperpolar.
  \item The $H_{o_2}$-action on $M_1$ is hyperpolar.
\end{enumerate}
\end{proposition}

\begin{proof}
Let $G$ and $K$ be as in Proposition~\ref{prop:PolCrit} and let $\p = \p_1 \oplus \p_2$, according to the splitting $M=M_1 \times M_2$. Let $d$ be the cohomogeneity of the $H$-action on~$M$. We may assume that $H$ is a closed subgroup of~$G$. Let $\nu \subset \p$ be the normal space $N_o(H \cdot o)$ and let $\nu_1 = \pi_1(\nu)$ and $\nu_2=\pi_2(\nu)$. We have $\nu_1 = N_{o_1} (H_{o_2} \cdot o_1)$ and $\nu_2 = N_{o_2} (H_{o_1} \cdot o_2)$. In case the actions of~$H$ on $M_1$ and on~$M_2$ are both transitive, we have that $\pi_i$ induces a linear isomorphism $\nu \to \nu_i$ for $i=1,2$.

Assume first the $H$-action on~$M$ is hyperpolar and indecomposable. By Lemma~\ref{lm:indechp}, both the $H$-action on~$M_1$ and the $H$-action on~$M_2$ are transitive, i.e.\ (i) and (ii) hold. After conjugating $H$, if necessary, we may assume $o=eK$. It follows from Proposition~\ref{prop:PolCrit} that~$\nu$, and hence both $\nu_1$ and $\nu_2$, contain a $d$-dimensional abelian subspace. Since $H$ acts transitively on~$M_1$ as well as on~$M_2$, the conjugacy classes of the isotropy subgroups $H_{o_1}$ and $H_{o_2}$ do not depend on the points $o_1$ or $o_2$. This shows that the $H_{o_1}$-action on~$M_2$ and the $H_{o_2}$-action on~$M_1$
are hyperpolar for arbitrary $o_1 \in M_1$ and $o_2 \in M_2$. In particular,
(iii) and (iv) hold.

Now assume (i)--(iv) hold. Since by assumption $H$ acts transitively on~$M_1$ and on~$M_2$, the conjugacy classes of the isotropy subgroups $H_{o_1}$ and $H_{o_2}$ do not depend on the points $o_1$ or $o_2$. Therefore, we may assume that $o$ is a regular point of the $H$-action on~$M$ by conjugating~$H$ in~$G$, if necessary. Then (iii) and (iv) still hold and $\nu_1$ and $\nu_2$ are $d$-dimensional abelian subspaces of~$\p_1$ and $\p_2$ by Proposition~\ref{prop:PolCrit}. This shows that $\nu \subset \p$ is a $d$-dimensional abelian subspace and it follows that the $H$-action on~$M$ is hyperpolar by Proposition~\ref{prop:PolCrit}.
\end{proof}

\begin{remark}\label{rem:NonSplt}
Note that if the equivalence in the statement of Proposition~\ref{prop:NonSplt} holds, then the $H$-action on~$M$, the $H_{o_1}$-action on~$M_2$ and the $H_{o_2}$-action on~$M_1$ all have the same cohomogeneity.
\end{remark}


\section{Hermann Actions}
\label{sec:Hermann}


The actions we will describe in this section were introduced by Hermann~\cite{hermann} as examples of variationally complete actions. It was shown by Conlon~\cite{conlon} that hyperpolar actions are variationally complete and much later by Gorodski and Thorbergsson~\cite{gt} that an isometric action on a compact symmetric space is variationally complete if and only if it is hyperpolar.

\begin{definition}\label{def:Her}
Let $M$ be a Riemannian symmetric space of the compact type. Let $G$ be the isometry group of~$M$. If $H \subset G$  is a locally symmetric subgroup of~$G$ then the action of~$H$ on~$M$ is called a \emph{Hermann action}.
\end{definition}

Let $G$ be a semisimple compact Lie group and assume $G=G_1 \times \dots \times G_n$ where the $G_i$ are simple. Let $\alpha$ be an involutive automorphism of~$G$. Then there is a self-inverse permutation $\pi = \pi_\alpha$ of the set $\{1, \dots, n\}$ and isomorphisms $\alpha_i \colon G_i \to G_{\pi(i)}$ such that
\begin{equation}\label{eq:InvAut}
\alpha(g_1, \dots, g_n) = (\alpha_{\pi(1)}(g_{\pi(1)}),\dots,\alpha_{\pi(n)}(g_{\pi(n)}))
\end{equation}
and we have $\alpha_i = \alpha_j^{-1}$ whenever $\pi(i) = j$.

\begin{remark}\label{rem:swap}
Let $G$ be a Riemannian symmetric space of Type~II, i.e.\ a simple compact Lie group, equipped with a biinvariant Riemannian metric. Then the action of $G \times G$ on~$G$ given by $(h,k) \cdot g = h \, g \, k^{-1}$ is conjugate to the action of $G \times G$ on~$G$ given by $(h,k) \cdot g = k \, g \, h^{-1}$. In fact, the inversion map $g \mapsto g^{-1}$ is an equivariant isometry.
\end{remark}

\begin{remark}\label{rem:autom}
Similarly, if $G$ is a Riemannian symmetric space of Type~II and $\alpha$ is an automorphism of~$G$, then the action of $G \times G$ on~$G$ given by $(h,k) \cdot g = h \, g \, k^{-1}$ is conjugate to the action of $G \times G$ on~$G$ given by $(h,k) \cdot g = \alpha(h) \, g \, \alpha(k)^{-1}$. Indeed, an equivariant isometry is given by the map $\alpha \colon G \to G$.
\end{remark}

\begin{proposition}\label{prop:indecHer}
An indecomposable Hermann action on a Riemannian symmetric space of the compact type is locally conjugate to one of the following actions:
\begin{enumerate}\def\theenumi{\roman{enumi}}

 \item\label{it:hkong} the action of~$H \times L^{n-1} \times K$ on~$L^n$, defined by
      \begin{align*}
      (h,&g_1,\dots,g_{n-1},k) \cdot (x_1,\dots,x_n) = \\ &= (h \, x_1 \, g_1^{-1},\; g_1 \, x_2 \, g_2^{-1},\; \dots\; ,\; g_{n-2} \, x_{n-1} \, g_{n-1}^{-1},\; g_{n-1} \, x_n \, k^{-1}),
      \end{align*}

 \item\label{it:hongk} the action of~$H \times L^{n-1}$ on~$L^{n-1} \times L/K$ obtained from~(\ref{it:hkong}) by reducing the last factor,

 \item\label{it:onhgk} the action of $L^{n-1}$ on $L/H \times L^{n-2} \times L/K$ obtained from~(\ref{it:hkong}) by reducing the first and the last factor,

 \item\label{it:sigma} the action of~$L^n$ on~$L^n$, defined by
      \begin{align*}
      (g_1,&\dots,g_n) \cdot (x_1,\dots,x_n) = \\ &= (g_1 \, x_1 \, g_2^{-1},\; g_2 \, x_2 \, g_3^{-1},\; \dots\; ,\; g_{n-1} \, x_{n-1} \, g_n^{-1},\; g_n \, x_n \, \sigma(g_1)^{-1}),
      \end{align*}

\end{enumerate}
where $L$ is a simply connected simple compact group, $H$ and $K$ are locally symmetric subgroups of~$L$ and $\sigma$ is an outer or trivial automorphism of~$L$. Conversely, the actions above are all indecomposable.
\end{proposition}

\begin{proof}
We may assume that $M$ is a simply connected Riemannian symmetric space of the compact type. Let $G$ be the connected component of the universal cover of the isometry group of~$M$. Let $\rho$ be an involutive automorphism of~$G$ such that $M = G/G^\rho$ up to coverings. We may assume the presentation $G/G^\rho$ is almost effective, i.e.\ the Lie algebra of~$G^\rho$ does not contain non-trivial ideals of~$\g$.

Furthermore, we may assume that the Hermann action on~$M$ is given by~$G^\tau$, where $\tau$ is an involutive automorphism of~$G$. Since the Hermann action is indecomposable, it follows that also the Lie algebra of~$G^\tau$ does not contain non-trivial ideals of~$\g$.

Define a graph as follows. The vertices are the simple factors $G_1, \dots, G_m$ of~$G$. Let $\pi_\rho$ and $\pi_\tau$ be the permutations associated with the involutive automorphisms~$\rho$ and $\tau$ as in~(\ref{eq:InvAut}). Two vertices $G_i$ and $G_j$, $i \neq j$, are connected by two edges if $\pi_\rho(i)=j$ and $\pi_\tau(i)=j$. Two vertices $G_i$ and $G_j$, $i \neq j$, are connected by one edge if either $\pi_\rho(i)=j$ or $\pi_\tau(i)=j$.  Two vertices $G_i$ and $G_j$ are connected by zero edges if $i=j$ or $\pi_\rho(i) \neq j \neq \pi_\tau(i)$.

If this graph is disconnected, then the Hermann action is decomposable; indeed, the Hermann action can then be written as a product action where the factors correspond to the connected components of the graph.

Now assume the action is indecomposable, i.e.\ the graph defined above is connected. Then it follows that all simple factors of~$G$ are isomorphic to a simply connected simple compact Lie group~$L$ and hence $G \cong L^m$. It remains to show that the $H$-action is locally conjugate to one of the actions (i) through~(iv). Since any vertex of the graph is connected with any other vertex by at most two edges, it is either a path graph, a cycle graph, or it consists of two vertices connected by two edges.

Let us first assume it is a path graph. By renumbering the nodes, we may assume that the two terminal vertices of the graph are $G_1$ and $G_m$ and that $G_i$ is connected with $G_{i+1}$ by one edge for $1 \le i \le m-1$.
We have case~(\ref{it:hkong}) with $m=2n$ if $\pi_\tau(1)=1$ and $\pi_\tau(m)=m$; we have case~(\ref{it:onhgk}) with $m=2n-2$ if $\pi_\rho(1)=1$ and $\pi_\rho(m)=m$. If either
$\pi_\tau(1)=1$ and $\pi_\rho(m)=m$ or $\pi_\rho(1)=1$ and $\pi_\tau(m)=m$, we have case~(\ref{it:hongk}) with $m=2n-1$, see also Remark~\ref{rem:swap}. Note that we may assume the isomorphisms $\rho_i$ with $\pi_\rho(i) \neq i$ and $\tau_i$ with $\pi_\tau(i) \neq i$ are all equal to the identity map of~$L$ by using Remark~\ref{rem:autom}.

By a similar argument, it follows that the Hermann action is locally conjugate to one of the actions as described in~(\ref{it:sigma}) with $m=2n$ in the case of a cycle graph or two vertices connected by two edges. By renumbering the nodes, we may assume that $G_i$ is connected with $G_{i+1}$ by an edge for $1 \le i \le m-1$ and that $G_m$ is connected with $G_1$. Using Remark~\ref{rem:autom}, we may assume that the isomorphisms $\rho_i, i=1,\dots,m,$ and $\tau_i, i=1,\dots,m-1$, are equal to the identity map of~$L$. By conjugation of the group acting, we may assume $\tau_m = \id_L$ if $\tau_m$ is an inner automorphism of~$L$.
\end{proof}

\begin{remarks}\label{rem:HerSigma}
The actions described in Proposition~\ref{prop:indecHer}~(\ref{it:sigma}) were called ``$\sigma$-actions'' in~\cite{hptt} in the special case $n=1$.
The irreducible (in the sense of Definition~\ref{def:Expand}) Hermann actions on symmetric spaces of compact type are the following cases in Proposition~\ref{prop:indecHer}:
Case~(ii) with~$n=1$, i.e.\ the Hermann actions on irreducible symmetric spaces of Type~I;
Case~(iii) with~$n=2$, i.e.\ the action of $\Delta G \subset G \times G$ on the product $G/K \times G/L$ of two irreducible symmetric spaces of Type~I with locally isomorphic isometry group;
Case~(iv) with~$n=1$, i.e.\ $\sigma$-actions on simple compact Lie groups.
The actions given in Case~(i) are never irreducible.
\end{remarks}

\begin{proposition}\label{prop:HermExp}
Let $M_1$ and $M_2$ be Riemannian symmetric spaces of the compact type and let $M = M_1 \times M_2$. Let $H$ be a closed connected subgroup of $G = G_1 \times G_2$, where $G_i$ is the connected component of the universal cover of the isometry group of~$M_i$. Assume the $\bar H$-action on~$\bar M$ is obtained from the $H$-action on~$M$ by expanding the factor~$M_2$. Then the $H$-action on~$M$ is a Hermann action if and only if the $\bar H$-action on~$\bar M$ is.
\end{proposition}

\begin{proof}
Assume the $H$-action on~$M$ is a Hermann action. Then there is an involutive automorphism~$\tau$ of~$G$ such that $H$ is the connected component of the fixed point set of~$\tau$. Furthermore, there is an involutive automorphism $\rho_2$ of~$G_2$ such that $M_2 = G_2 / K_2$, where $K_2$ agrees up to components with the fixed point set of~$\rho_2$.
Let $\bar M = M_1 \times G_2$. The connected component of the universal cover of the isometry group of~$\bar M$ is $\bar G = G \times G_2$. Define $\bar\tau \in \Aut(\bar G)$ by $\bar\tau(g,g_2) = (\tau(g),\rho_2(g_2))$ for $g \in G$, $g_2 \in G_2$. This shows that $\bar H = H \times K_2$ is a symmetric subgroup of~$\bar G$.

Conversely, if $\bar H = H \times K_2$ agrees up to components with the fixed point set of an involutive automorphism of $\bar G \cong G \times G_2$ and $K_2$ is a symmetric subgroup of~$G_2$, then this automorphism restricts to an involution of~$G$.
\end{proof}


\section{Transitive Actions}
\label{sec:Trans}


Let $G$ be a Lie group and let $G', G'' \subseteq G$ be two closed subgroups. Following Oni\v{s}\v{c}ik~\cite{oniscik}, the triple $(G,G',G'')$ is said to be a \emph{decomposition} of~$G$ if we have $G = G' \cdot G''$, i.e.\ if
every element $g \in G$ can be written as $g = g'\,g''$, where $g' \in G'$ and $g'' \in G''$. A decomposition $(G,G',G'')$ is said to be a \emph{proper decomposition} if the Lie algebras of both $G'$ and $G''$ are proper subalgebras of the Lie algebra of~$G$. All proper decompositions of connected simple compact Lie groups have been determined in \cite[Theorem~4.1]{oniscik}. We reproduce the result in our Table~\ref{t:Trans}. Let us make some remarks on the table. The problem of finding all decompositions of a connected Lie group can be formulated entirely on the Lie algebra level and accordingly the table has to be interpreted on the Lie algebra level. It depends only on the conjugacy class of~$G'$ and~$G''$ in~$G$ if $(G,G',G'')$ is a decomposition.

One can immediately make the following observations from the table: The only simple compact Lie groups having proper decompositions are those of types $\LA_{2n-1}$, $\LB_3$, $\LD_n$. Furthermore, if $(G,G',G'')$ is a proper decomposition of the compact Lie group~$G$, then $G'$ and $G''$ are products all of whose normal subgroups are of type $\LA_n$, $\LB_n$, $\LC_n$, $\LD_3$, $\LG_2$, or one-dimensional.

\begin{proposition}\label{prop:TransDec}
Let $G$ be a connected simple compact Lie group. Let $H \subset G \times G$ be a closed connected subgroup whose action on~$G$ as defined in (\ref{eq:GAct}) is transitive. Then $H = G' \times G''$ where $(G, G',G'')$ is a decomposition of~$G$.
\end{proposition}

\begin{proof}
We may assume $H$ is not of the form $G \times H_2$ or $H_1 \times G$ for $H_1,H_2 \subseteq G$, since $(G,G,H_2)$ and $(G,H_1,G)$ are decompositions of~$G$. Then $H$ is contained in a maximal connected subgroup of $G \times G$, i.e.\ $H$ is contained in a subgroup of the form $\{(g,\sigma(g)) \mid g \in G \}$, where $\sigma \in \Aut(G)$, or in a subgroup of the form $G \times G''$ or $G' \times G$, where $G',G'' \subset G$ are closed connected subgroups, see \cite[Theorem~2.1]{hyperpolar} or \cite[Theorem~15.1]{dynkin1}. The first kind of subgroups does not act transitively on~$G$, see~\cite[Section~3.2]{hyperpolar}. Hence $G$ is contained in a subgroup of the latter kind. Define $G' = \pi_1(H)$ and $G'' = \pi_2(H)$. Then $(G,G',G'')$ is a proper decomposition of the simple compact Lie group~$G$. If $H = G' \times G''$, we are done.
Thus assume now $H \neq G' \times G''$, where $(G,G',G'')$ is a proper decomposition of~$G$. It follows from \cite[Theorem~2.1]{hyperpolar} that the Lie algebras $\g'$ and $\g''$ contain an isomorphic ideal.
The only proper decomposition of a simple compact Lie group where this is the case is $(\SO(8),\Spin(7),\SO(7))$, see Table~\ref{t:Trans}. It follows that $H$ is a diagonal subgroup of $\Spin(7) \times \SO(7)$.
However, a Lie group of dimension~$21$ cannot act transitively on~$\SO(8)$. \end{proof}

\begin{table}
\begin{tabular}{*{5}{|c}|}
\hline
\str {No.} & $G'$ & $G$ & $G''$ & $(G' \cap G'')^0$ \\
\hline\hline
\str $1$ & $\Sp(n)$ & $\SU(2n)$ & $\SUxU{2n-1}{1}$ & $\Sp(n-1)\times\U(1)$ \\

 $$ &  & & $\SU(2n-1)$ & $\Sp(n-1)$ \\
\hline
\str $2$ & $\SO(2n-1)$ & $\SO(2n)$ &  $\U(n)$ & $\U(n-1)$ \\

 $$ &  & & $\SU(n)$ & $\SU(n-1)$ \\

\str $$ & $\Spin(7)$ & $\SO(8)$ &  $\SO(6)\times\SO(2)$ & $\U(3)$ \\
 $$ &  & &  $\SO(6)$ & $\SU(3)$ \\
\hline
\str $3$ & $\SO(4n-1)$ & $\SO(4n)$ &  $\Sp(n)\cdot\Sp(1)$ & $\Sp(n-1)\cdot\Sp(1)$ \\

 $$ &  & & $\Sp(n)\cdot\U(1)$ & $\Sp(n-1)\cdot\U(1)$ \\

 $$ &  & & $\Sp(n)$ & $\Sp(n-1)$ \\
\str $$ & $\Spin(7)$ & $\SO(8)$ &  $\SO(5)\times\SO(3)$ & $\Sp(1) \cdot \Sp(1)$ \\
 $$ &  & &  $\SO(5)\times\SO(2)$ & $\Sp(1) \cdot \U(1)$ \\
 $$ &  & &  $\SO(5)$ & $\Sp(1)$ \\
\hline
\str $4$ & $\LG_2$ & $\SO(7)$ &  $\SO(6)$ & $\SU(3)$ \\
\hline
\str $5$ & $\LG_2$ & $\SO(7)$ &  $\SO(5)\times\SO(2)$ & $\U(2)$ \\

 $$ &  & & $\SO(5)$ & $\SU(2)$ \\
\hline

\str $6$ & $\Spin(7)$ & $\SO(8)$ &  $\SO(7)$ & $\LG_2$ \\
\hline
\str $7$ & $\Spin(9)$ & $\SO(16)$ &  $\SO(15)$ & $\Spin(7)$ \\
\hline
\end{tabular}
\bigskip
\caption{Transitive actions.} \label{t:Trans}
\end{table}


\section{Two Irreducible Factors}
\label{sec:TwoFactors}


The purpose of this section is to prove the following.

\begin{theorem}\label{thm:hyptwof}
Let $M_1$ and $M_2$ be irreducible Riemannian symmetric spaces of the compact type. Assume there is an indecomposable hyperpolar action of a compact Lie group~$H$ on~$M_1 \times M_2$ of cohomogeneity greater than one. Then the $H$-action is orbit equivalent to a Hermann action.
\end{theorem}

\begin{proof}
By expanding the two factors $M_1$ and $M_2$, if necessary, we may assume that both factors $M_1$ and $M_2$ are symmetric spaces of Type~II, i.e.\ simple compact Lie groups. We may assume the compact Lie group~$H$ acts  almost effectively and hyperpolarly, but not transitively, on~$M = M_1 \times M_2$ in such a way that the $H$-actions on both $M_1$ and $M_2$ are transitive. By Proposition~\ref{prop:TransDec} it follows that both the $H$-action on~$M_1$ and the $H$-action on~$M_2$ are given by (not necessarily proper) decompositions of~$M_1$ and~$M_2$. We start with the case where both decompositions are proper. Let us remark that in order to study the action on $M_1 \times M_2$, we may reduce the factors $M_1$, $M_2$, or both, if possible; by Theorem~\ref{thm:str} this makes no difference for the hyperpolarity of the action and it does not change the cohomogeneity. In the following, we will use this fact wherever it is convenient.

We may assume that the $H$-action on~$M = M_1 \times M_2$ is not a product action (not even locally) and hence that there is a nonzero ideal of~$\h$ such that the corresponding connected subgroup of~$H$ acts almost effectively on~$M_1$ as well as on~$M_2$. We call a maximal such group a \emph{diagonal factor} of the two actions.

\begin{lemma}\label{lm:RankOneDone}
Let $M_1$ and $M_2$ be Riemannian symmetric spaces of the compact type.
Assume there is an indecomposable hyperpolar action on $M_1 \times M_2$. Then the cohomogeneity of the action is less or equal $\min(\rk(M_1),\rk(M_2))$.
In particular, if one of the spaces $M_1$ or $M_2$ is a rank-one symmetric space, then the action is of cohomogeneity one.
\end{lemma}

\begin{proof}
Assume there is an indecomposable hyperpolar action on $M_1 \times M_2$ of cohomogeneity~$d$. By Proposition~\ref{prop:NonSplt} it follows that there are hyperpolar actions on $M_i$, $i=1,2$ of cohomogeneity~$d$. Thus (\ref{eq:dim}) implies that $d \le \min(\rk(M_1),\rk(M_2))$.
\end{proof}

\subsection{Both decompositions are proper}\label{subsec:pp}

Inspection of Table~\ref{t:Trans} shows that a diagonal factor is either semisimple or the direct product of a semisimple and a one-dimensional Lie group if both decompositions are proper.
We start by considering simple diagonal factors. We will indicate the type of action we are considering by referring to the numbering of Table~\ref{t:Trans}, e.g.\ if the diagonal factor is $\LG_2$, then $4-5$ refers to the action of~$\LG_2$ on $(\SO(7)/\SO(6)) \times (\SO(7) / \SO(5) \SO(2))$. In cases where an ambiguity can arise, a~${\rm \Delta}$ is put in front of the diagonal factor.

\subsubsection*{Diagonal factors of type $\LA_n, n \ge 2$} In Table~\ref{t:Trans}, only the decompositions 1, 2 and 4 have simple factors of type~$\LA_n, n \ge 2$.

\paragraph{$1-1$} Consider the action of $\U(1) \times \SU(2n-1) \times \U(1)$ on $(\SU(2n) / \Sp(n))^2$ for $n \ge 2$. By Proposition~\ref{prop:NonSplt} and Table~\ref{t:Trans}, this action is hyperpolar only if the action of $\Sp(n-1) \times \U(1)$ on $\SU(2n) / \Sp(n)$ is. However, this action is a subaction of the $\SUxU{2n-2}{2}$-action on $\SU(2n) / \Sp(n)$, which is of cohomogeneity one~\cite[Theorem~B]{hyperpolar}. Thus, by Lemma~\ref{lm:maxsubg}, it is of cohomogeneity one or it is not hyperpolar.

\paragraph{$1-2,\;2-2,\;2-4$} By Lemma~\ref{lm:RankOneDone}, these actions are of cohomogeneity one if they are hyperpolar.

\paragraph{$4-4$} By Proposition~\ref{prop:NonSplt} and Table~\ref{t:Trans}, the action of $\SO(6)\times \LG_2^2$ on~$\SO(7)^2$ is hyperpolar only if the $\SU(3) \times \LG_2$-action on~$\SO(7)$ is hyperpolar. This action is a subaction of the $\U(3) \times \LG_2$-action on~$\SO(7)$, which is of cohomogeneity one~\cite[Theorem~B]{hyperpolar}. Thus, by Lemma~\ref{lm:maxsubg}, it is of cohomogeneity one or it is not hyperpolar.

\subsubsection*{Diagonal factors of type $\LB_n, n \ge 2$}

The only decompositions with simple factors of type~$\LB_n, n \ge 2$, are 2, 3, 5, 6, 7. Note that decomposition~1 with $n=2$ is (locally) the same as decomposition~2 with $n=3$, using the isomorphism $\SU(4) \cong Spin(6)$. Therefore, we do not need to consider decomposition~1 here.

\paragraph{$2-2$} By Proposition~\ref{prop:NonSplt} and Table~\ref{t:Trans}, the action of~$\SO(2n-1)$ on $(\SO(2n) / \U(n))^2, n \ge 3$, is hyperpolar if and only if the action of $\U(n-1)$ on $\SO(2n) / \U(n)$ is. This action is a subaction of the $\SO(2n-2) \times \SO(2)$-action on $\SO(2n) / \U(n)$, which is of cohomogeneity~one. Hence it follows from Lemma~\ref{lm:maxsubg} that the $\SO(2n-1)$-action is non-hyperpolar if it is not of cohomogeneity~one.

\paragraph{$2-3$} Assume the action of~$\SO(4n-1) \times \Sp(n) \cdot \Sp(1)$ on $(\SO(4n) / \U(2n)) \times \SO(4n)$ is hyperpolar. This leads to a contradiction with condition~(\ref{eq:dim}) if $\sf n \ge 3$. Since $\Sp(2) \cdot \Sp(1)$ is a locally symmetric subgroup of $\SO(8)$, cf.~\cite[Proposition~3.3]{hyperpolar}, we obtain a contradiction by Lemma~\ref{lm:RankOneDone} in case $\sf n = 2$.

\paragraph{$2-5,\;2-6$} It follows from Lemma~\ref{lm:RankOneDone} that these actions are of cohomogeneity one if they are hyperpolar.

\paragraph{$2-7$} Here we have to consider two different actions, depending on whether the diagonal factor is $\SO(15)$ or $\Spin(9)$. The action of $\Spin(9) \times {\rm \Delta} \SO(15)$ on $(\SO(16) / \U(8)) \times \SO(16)$ is not hyperpolar by condition~(\ref{eq:dim}). The action with diagonal factor ${\rm \Delta} \Spin(9)$ is of cohomogeneity one or not hyperpolar by Lemma~\ref{lm:RankOneDone}.

\paragraph{$3-3$} Consider the action of $(\Sp(n) \cdot \Sp(1)) \times \SO(4n-1) \times (\Sp(n) \cdot \Sp(1))$ on $\SO(4n)^2$, for $n \ge 2$. This action is not hyperpolar by condition~(\ref{eq:dim}).

\paragraph{$3-6$} By Lemma~\ref{lm:RankOneDone}, this action is of cohomogeneity one if it is hyperpolar.

\paragraph{$3-7$} The action of $(\Sp(4) \cdot \Sp(1)) \times \SO(15) \times \Spin(9)$ on $\SO(16)^2$ is not hyperpolar by condition~(\ref{eq:dim}).

\paragraph{$5-5$} Consider the action of $\SO(5) \times \SO(2)^2 \times \LG_2^2$ on~$\SO(7)^2$. By Proposition~\ref{prop:NonSplt} and Table~\ref{t:Trans}, this action is hyperpolar if and only if the action of~$\U(2) \times \LG_2$ on $\SO(7)$ is. This action is a subaction of the $\LG_2^2$-action on~$\SO(7)$, which is of cohomogeneity one. Hence this action is of cohomogeneity one if it is hyperpolar by Lemma~\ref{lm:RankOneDone}.

\paragraph{$6-6$} It follows from Lemma~\ref{lm:RankOneDone} that this action is of cohomogeneity one if it is hyperpolar.

\paragraph{$7-7$} There are two actions to consider here. First, it follows from Lemma~\ref{lm:RankOneDone} that the ${\rm \Delta}\Spin(9)$-action on $(\eS^{15})^2$ is of cohomogeneity one if it is hyperpolar. Second, we have the action of ${\rm \Delta}\SO(15) \times \Spin(9)^2$ on~$\SO(16)^2$. This action is not hyperpolar by condition~(\ref{eq:dim}).

\subsubsection*{Diagonal factors of type $\LC_n, n \ge 3$} These only occur in decompositions 1 and 3.

\paragraph{$1-1,\;1-3,\;3-3$} It follows from Lemma~\ref{lm:RankOneDone} that these actions are of cohomogeneity one if they are hyperpolar.

\subsubsection*{Diagonal factors of type $\LG_2$}
Diagonal factors of type $\LG_2$ only occur in decompositions~4 and~5.

\paragraph{$4-4,\;4-5$} It follows from Lemma~\ref{lm:RankOneDone} that these actions are of cohomogeneity one if they are hyperpolar.

\paragraph{$5-5$} Assume the Lie group $\LG_2$ acts hyperpolarly on $(\SO(7) / \SO(5) \times \SO(2))^2$. It then follows from Lemma~\ref{lm:RankOneDone} that the cohomogeneity is at most~$4$. Counting dimensions leads to a contradiction.

\subsubsection*{Diagonal factors which are non-simple or contain normal subgroups of rank one}
In case the diagonal factor is one-dimensional, the action on $M_1 \times M_2$ is transitive, as can be seen from Table~\ref{t:Trans}, thus all non-transitive subactions of actions whose diagonal factor contains an abelian normal subgroup are also subactions of the actions considered above. Thus by Lemma~\ref{lm:maxsubg} they are of cohomogeneity one or not hyperpolar.
A similar argument applies to diagonal factors of rank one and those containing more than one simple ideal.

\subsection{One decomposition proper, the other non-proper}
\label{subsec:pnp}

We assume here that $M$ is the direct product of two simple compact Lie groups $M = G_1 \times G_2$. Furthermore, since one decomposition is non-proper, it follows that $H$ contains a simple factor which is locally isomorphic to one of $G_1$ or $G_2$, say~$G_1$. Since the $H$-action on~$M$ is indecomposable, it follows that this factor is contained in a diagonal factor of the action. Hence we may assume we are in the following situation.
The group $G_1 \subset G_2$ is a proper closed subgroup, $H_1 \subset G_1$ and $H_2 \subset G_2$ are proper closed subgroups, $Z$ is the centralizer of~$G_1$ in~$G_2$ and $G_2 = (G_1 Z) \cdot H_2$ is a proper decomposition of~$G_2$.
An action of $H_1 \times Z \times H_2 \times G_1$ on~$M$ is defined as follows:
\begin{align}\label{eq:pnp}
(h_1,z,h_2,g) \cdot (x,y) = (h_1 \, x \, g^{-1},\; g \, z \, y \, h \,_2^{-1}),
\end{align}
where $(h_1,z,h_2,g) \in H_1 \times Z \times H_2 \times G_1$ and $(x,y) \in G_1 \times G_2$.
Furthermore, we may assume $H$ is a closed subgroup of $H_1 \times Z \times H_2$ such that the restriction of the action of $H_1 \times Z \times H_2 \times G_1$ on~$M$ to $H \times G_1$ is hyperpolar.

\begin{lemma}\label{lm:pnp}
The action of $H \times G_1$ on $G = G_1 \times G_2$ as defined by~(\ref{eq:pnp}) and the action of $H$ on~$G_2$, defined by
\begin{align}\label{eq:pnp2ndf}
(h_1,z,h_2) \cdot y = h_1 \, z \, y \, h_2^{-1},
\end{align}
where $(h_1,z,h_2) \in H_1 \times Z \times H_2$, $y \in G_2$, have the same cohomogeneity. Furthermore, if the action of~$H \times G_1$ on~$G$ is hyperpolar, then the $H$-action on~$G_2$ is, too.
\end{lemma}

\begin{proof}
The second part of the statement follows from Proposition~\ref{prop:NonSplt}~(iii) if one chooses $o_1$ to be the identity element of~$G_1$. For the first part, see Remark~\ref{rem:NonSplt}.
\end{proof}

In Examples~\ref{eg:telescope} below, two such actions are exhibited.

\begin{lemma}\label{lm:pnpclass}
Let $G_1$ and $G_2$ be simple compact Lie groups such that $G_1 \subset G_2$ is a proper closed subgroup. Let $G = G_1 \times G_2$ be endowed with the biinvariant Riemannian metric induced by an $\Ad(G)$-invariant inner product on the Lie algebra of~$G$.
Let $H$ be a closed subgroup of $H_1 \times Z \times H_2$ where $H_1 \subset G_1$ and $H_2 \subset G_2$ are proper closed subgroups, where $Z$ is the centralizer of~$G_1$ in~$G_2$ and where $G_2 = (G_1 Z) \cdot H_2$ is a proper decomposition. Assume the $H \times G_1$-action on~$M$ as defined by~(\ref{eq:pnp}) is hyperpolar and indecomposable. Then it is of cohomogeneity one.
\end{lemma}

\begin{proof}
We will mostly use Lemma~\ref{lm:pnp}.
It follows from the lemma that the action of~$H$ on~$G_2$ is hyperpolar. By Lemma~\ref{lm:maxsubg} we may assume $H = H_1 \times Z \times H_2$. Such hyperpolar actions have been classified in~\cite[Subsection~2.4.5]{hyperpolar}. There we were only interested in maximal non-transitive subgroups and to prove the proposition we will have to slightly refine this classification.
As in~\cite{hyperpolar}, we will go through the rows of Table~\ref{t:Trans} case-by-case, using the following numbering. The various cases are denoted as 1.a), 1.b), $\dots$, 7.a), 7.b), where the number refers to the number in Table~\ref{t:Trans} and where the letter a) means $G'=G_1Z, G''=H_2$, the letter b) means $G'=H_2, G''=G_1Z$, cf.~\cite{hyperpolar}. We do not need to consider the cases where $G_1$ is of rank one, since in these cases the $H$-action on~$G_2$ is transitive, as can be seen from Table~\ref{t:Trans}.

We may also ignore those cases where $(G_2,H_2)$ is a symmetric pair such that $G_2/H_2$ is a symmetric space of rank one, since in this case hyperpolar actions are necessarily of rank one. This applies to Cases 1.a), 2.b), 3.b), 4.a), 6.a) 6.b) and 7.a).

\paragraph{1.b)} Let $G_1 = \SU(2n-1)$, $G_2=\SU(2n)$, $Z=\U(1)$, $H_2=\Sp(n)$. Let $o_2$ be the identity element of~$G_2$. Consider the intersection action of~$H_{o_2}$ on~$G_1$. It follows from Table~\ref{t:Trans} that this intersection action is given by the action of~$H_1 \times \Sp(n-1)$ on~$\SU(2n-1)$. This is a subaction of the action of~$H_1 \times \SUxU{2n-2}1$ on~$\SU(2n-1)$. If the former action is hyperpolar, then it follows from Lemma~\ref{lm:maxsubg} that the two actions are orbit equivalent. However, we have $\rk(\SU(2n-1) / \SUxU{2n-2}1) = 1$ and it follows that in this case the action is of cohomogeneity one.

\paragraph{2.a)} The case where $G_1 = \SO(2n-1)$, $G_2=\SO(2n)$, $Z=\{1\}$, $H_2=\U(n)$ can be treated in an analogous fashion as Case~1.b).

\paragraph{3.a)} Let $G_1 = \SO(4n-1)$, $G_2=\SO(4n)$, $Z=\{1\}$, $H_2=\Sp(n) \cdot \Sp(1)$. Let $o_2$ be the identity element of~$G_2$. Consider the intersection action of~$H_{o_2}$ on~$G_1$. It follows from Table~\ref{t:Trans} that this intersection action is given by the action of~$H_1 \times (\Sp(n-1)\times \SO(3))$ on~$\SO(4n-1)$. This action is a subaction of the action of~$H_1 \times (\SO(4n-4)\times \SO(3))$ on~$\SO(4n-1)$ and hence it is at most of cohomogeneity three if it is hyperpolar. By the main result of~\cite{hyperpolar}, if it is not of cohomogeneity one, we may assume $H_1 = \SO(4n-\ell-1)\times \SO(\ell)$, $2 \le \ell < 2n$, possibly replacing $H$ by a larger, orbit equivalent group. Now consider the $H$-action on~$G_2$ as given in~(\ref{eq:pnp2ndf}).
Then the $H$-action on $G_2$ is a subaction of the
action of $\SO(4n-\ell-1)\times \SO(\ell+1) \times \Sp(n) \cdot \Sp(1)$ on~$\SO(4n)$, which is not hyperpolar~\cite[Subsection~2.3.2]{hyperpolar}.
Thus the $H$-action on~$G_2$ is not hyperpolar by Lemma~\ref{lm:maxsubg}.

\paragraph{4.b)} See Case 4.b) in \cite[p.~601]{hyperpolar}.

\paragraph{5.a)} The maximal subgroups of~$\LG_2$ are $\SO(4)$, $\SU(3)$ and a group of type~$\LA_1$, cf.~Case~5.a) in \cite[p.~601]{hyperpolar}. The groups $\SO(4)$ and $\LA_1$ are excluded by~(\ref{eq:dim}).  The subgroup $\SU(3)$ is contained in~$\SO(6)$, which acts with cohomogeneity one on $\SO(7)/(\SO(5)\times\SO(2))$.

\paragraph{5.b)} Let $o_2$ be the identity element of~$G_2$. Consider the $H_{o_2}$-action on~$G_1$. This action is given by the action of $H_1 \times \U(2)$ on~$\SO(5)$, where $H_1 \subset \SO(5)$ is a proper closed subgroup.
The maximal connected subgroups of $\SO(5)$ are $\SO(4)$, $\SO(3) \times \SO(2)$ and a group of type~$\LA_1$. The actions of $\SO(4) \times \U(2)$ and $(\SO(3) \times \SO(2)) \times \U(2)$ are subactions of cohomogeneity one actions, thus they are of cohomogeneity one if they are hyperpolar by Lemma~\ref{lm:maxsubg}. The subgroup of type~$\LA_1$ can be excluded by a dimension count.

\paragraph{7.b)} The irreducible maximal connected subgroups of $G_1=\SO(15)$ are excluded by a dimension count, cf.~\cite[Appendix]{hyperpolar}, the reducible ones are contained in reducible subgroups~$\SO(k) \times \SO(16-k), k=2,\dots,8$ of~$\SO(16)$. It was shown in~\cite[Subsection~2.4.1]{hyperpolar} that the action of $\Spin(9) \times (\SO(k) \times \SO(16-k))$, $k=2,\dots,8$ on~$\SO(16)$ is hyperpolar only if $k=2$, in which case the cohomogeneity is one.
\end{proof}

\subsection{Both decompositions are non-proper}\label{subsec:npnp}

With the same argument as in the beginning of Subsection~\ref{subsec:pnp}, the non-proper factor in each decomposition must be contained in the diagonal factor, forcing the two factors $M_1$ and $M_2$ to be locally isomorphic. Thus, up to coverings, we are looking at actions of the following form. Let $G$ be a simple compact Lie group and let $H \subseteq G \times G$ be a closed subgroup. Then $H \times G$ acts on $G \times G$ by the rule $(h_1,h_2,g) \cdot (x,y) = (h_1 \, x \, g^{-1},\; g \, y \, h_2^{-1})$, where $(h_1,h_2) \in H$, $g \in G$, $(x,y) \in G \times G$. Note that this action is obtained from the $H$-action on~$G$ as defined in~(\ref{eq:GAct}) by expanding~$G$. By Theorem~\ref{thm:str}, the $H$-action on the simple compact Lie group~$G$ is hyperpolar. Therefore, it is orbit equivalent to a Hermann action or of cohomogeneity one by the main result of~\cite{hyperpolar}. It follows from Proposition~\ref{prop:HermExp} and Lemma~\ref{lm:orbeqex} that the $H \times G$-action on~$G \times G$ is orbit equivalent to a Hermann action.  We have completed the proof of Theorem~\ref{thm:hyptwof}.
\end{proof}


\section{Arbitrarily Many Factors}
\label{sec:ManyFactors}


\begin{proof}[Proof of Theorem~\ref{main:Hermann}]
Let $M$ be a Riemannian symmetric space of compact type. Assume the compact Lie group $H$ acts isometrically on~$M$ such that the $H$-action on~$M$ is indecomposable and hyperpolar with cohomogeneity greater than one. It follows from Theorem~\ref{thm:str} that the group lift of the $H$-action on~$M$ is also indecomposable and hyperpolar with the same cohomogeneity. By Lemma~\ref{lm:orbeqex} and Proposition~\ref{prop:HermExp}, it suffices to show that the group lift of the $H$-action on~$M$ is orbit equivalent to a Hermann action. Hence, after replacing the $H$-action on~$M$ by its group lift, we may assume that $M$ is a compact Lie group with a biinvariant Riemannian metric. By Remark~\ref{rem:hypliealg}, we may assume that $M$ is simply connected. Thus it is of the form $M = M_1 \times M_2 \times \dots \times M_n$, where the $M_i$ are simply connected irreducible Riemannian symmetric spaces of Type~II.

It remains to show that the $H$-action on~$M$ is orbit equivalent to a Hermann action in this situation. If $M$ is irreducible, i.e.\ in case $n=1$, this follows from the main result of~\cite[Theorem~A]{hyperpolar}. For the case $n=2$, this was shown in Section~\ref{sec:TwoFactors}. Hence we may assume $n \ge 3$.

Let now $i,j \in \{1, \dots, n\}$ such that $i \neq j$. Let $I := \{i,j\}$ and let $J := \{1, \dots, n\} \setminus I$. Let $M_I = \prod_{k \in I} M_k$ and $M_J = \prod_{k \in J} M_k$. Let $o_I \in M_I$ and let $o_J \in M_J$. Consider the intersection action of $H_{o_J}$ on $M_I$. By Proposition~\ref{prop:NonSplt}, this intersection action is hyperpolar. By assumption, the space $M_I$ is a Riemannian product of two symmetric spaces of Type~II. By Remark~\ref{rem:NonSplt}, the intersection action has the same cohomogeneity as the $H$-action on~$M$, in particular, the cohomogeneity is~$\ge 2$. Thus it follows from Theorem~\ref{thm:hyptwof} that the intersection action on~$M_I$ is orbit equivalent to a Hermann action. Furthermore, we have that $M_i \cong M_j$. Since this argument applies to arbitrary pairs $i \neq j$, it follows that there is a simply connected simple compact Lie group~$L$ such that $L \cong M_1 \cong M_2 \cong \dots \cong M_n$.

Now let $k \in \{1, \dots, n\}$ such that $j \neq k \neq i$. Consider the intersection action on $M_j \times M_k$. By the same argument as above, this intersection action is also orbit equivalent to a Hermann action. Now consider the intersection action on~$M_j$. This action is at the same time an intersection action of the $H$-action on~$M$ as well as of both its intersection actions on $M_i \times M_j$ and $M_j \times M_k$.  It follows from Lemma~\ref{lm:orbeqex} that the intersection actions on $M_i \times M_j$ and on $M_j \times M_k$ are orbit equivalent, because they are both orbit equivalent to the action obtained by expanding the factor $M_j$ in the intersection action on~$M_j$. Since this argument applies to arbitrary subsets $\{i,j,k\} \subseteq \{1, \dots, n\}$ of cardinality~$3$, it follows that all the intersection actions on the simple factors~$M_i \cong L$ are orbit equivalent. Let $Q \subset L \times L$ be a connected locally symmetric subgroup whose action on~$L$ is orbit equivalent to all of these actions. After replacing $H$ by a conjugate subgroup in the isometry group of~$M$, if necessary, and using suitable identifications such that $L=M_1=\dots=M_n$, we may assume that all intersection actions on products of two simple factors $M_i \times M_j$ have the same connected components of orbits as the action obtained from the $Q$-action on~$L$ by expanding a simple factor, see also Remarks~\ref{rem:swap} and~\ref{rem:autom}.

To show that the $H$-action on~$M$ is orbit equivalent to the action on~$L^n$ obtained from the $Q$-action on~$L$ by $n$~times expanding a simple factor, it now suffices to prove that the connected components of the orbits of any indecomposable hyperpolar action on a semisimple compact Lie group is already uniquely determined by the orbits of the intersection actions on all products of two simple factors $M_i \times M_j$, where $i, j \in \{1, \dots, n\}$, $i \neq j$. Let $o_i \in M_i$ for $i=1,\dots,n$. Assume the point $o=(o_1,\dots,o_n)$ lies in a regular orbit of the $H$-action on~$M$. Let $\Sigma$ be the unique section of the $H$-action containing the point~$o$ and let $\nu = T_o\Sigma$. For $i=1,\dots,n$, let $\nu_i = \pi_i(\nu)$. It follows from Proposition~\ref{prop:NonSplt} that for every $i,j \in \{1, \dots, n\}$ there is a linear isomorphism $\Phi_{ij} \colon \nu_i \to \nu_j$ such that $\{X + \Phi_{ij}(X) \mid X \in \nu_i \}$ is the tangent space to the section of the intersection action on~$M_i \times M_j$ at the point $(o_i,o_j)$.  It follows that every vector of~$\nu$ is contained in the set $\{ X_1 + \Phi_{12}(X_1)+\dots+\Phi_{1n}(X_1) \mid X_1 \in \nu_1 \}$. By a dimension count, it follows that this set is actually equal to~$\nu$. This argument shows that the $H$-action on~$M$ has the same principal orbits as the action on~$L^n$ obtained from the $Q$-action on~$L$ by $n$~times expanding a simple factor~$L$. Thus the two actions are orbit equivalent.
\end{proof}


\section{Examples of Cohomogeneity One Actions}
\label{sec:Examples}


In this section we give some examples of indecomposable cohomogeneity one actions. Some of them can be found in~\cite[p.~21]{drk}.

\begin{example}
Let $G=\SO(n+1)$ and let $K=\SO(n)$. Then $G/K=\eS^n$ is the $n$-sphere and the $K$-action on~$G/K$ is an example of a Hermann action of cohomogeneity one. By expanding the factor $\eS^n$ we obtain the $K \times K$-action on~$G$. If we further expand the factor~$G$ we obtain the action of $K \times G \times K$ on~$G \times G$ given by $(k_1,g,k_2) \cdot (g_1,g_2) = (k_1 \, g_1 \, g^{-1}, \; g \, g_2 \, k_2^{-1})$.
Since $K \subset G$ is a symmetric subgroup, we can reduce each of the $G$-factors. If we do this for both factors, we obtain the action of $\SO(n)$ on $\eS^n \times \eS^n$, see~\cite[Example~1]{drk}, where the projection actions of $\SO(n)$ on both factors are conjugate.
\end{example}

Of course, like in the example above, one may take any cohomogeneity one action on a simple compact Lie group as given in~\cite[Theorem~B]{hyperpolar} and expand factors.  This process may be repeated arbitrarily many times.
However, there are many indecomposable cohomogeneity one actions which are not given by this construction.  We will illustrate this below.

\begin{examples}\label{eg:SpinEight}
We shall give two examples for cohomogeneity one actions arising from the triality automorphism of~$\Spin(8)$.
\begin{enumerate}
\item
The compact Lie group $\Spin(8)$ has three pairwise inequivalent irreducible real $8$-di\-men\-sion\-al representations $\varrho_0$, $\varrho_1$, $\varrho_2$ and we may use them to define a Lie group homomorphism
\[
{\Spin(8) \to \SO(8) \times \SO(8) \times \SO(8)}, \quad g \mapsto (\varrho_0(g),\varrho_1(g),\varrho_2(g)).
\]
This defines an isometric action of $\Spin(8)$ on~$M := \eS^7 \times \eS^7 \times \eS^7$.
The conjugation classes of the principal isotropy groups of the three representations $\varrho_0$, $\varrho_1$, and $\varrho_2$ acting on $\eS^7$ are the three conjugacy classes of the subgroup $\Spin(7) \subset \Spin(8)$. The action of $\Spin(8)$ on~$M$ is of cohomogeneity one and indecomposable. This can be seen as follows. Let $M_1 = \eS^7 \times \eS^7$ and let $M_2 = \eS^7$. (It does not matter which factors are grouped together in~$M_2$, since there are outer automorphisms of~$\Spin(8)$ which correspond to permutations of the three representations  $\varrho_0$, $\varrho_1$, $\varrho_2$.)
Then the $\Spin(8)$-action is transitive on~$M_1$ with isotropy group~$\LG_2$. Since $\LG_2$ acts with cohomogeneity one on~$M_2 = \eS^7$, it follows that the $\Spin(8)$-action on~$M$ is of cohomogeneity one. Using Proposition~\ref{prop:NonSplt} and the fact that $\Spin(8)$ acts transitively on both factors~$M_1$ and $M_2$, it follows that the $\Spin(8)$-action on~$M$ is indecomposable cf.~\cite{drk}.
\item
The following is a modification of Example~\ref{eg:SpinEight}.
Consider the same subgroup of~$\SO(8) \times \SO(8) \times \SO(8)$ as defined above and let it act on
\[
M := \eS^7 \times \eS^7 \times \Gr_2(\R^8) = \frac{\SO(8)}{\SO(7)} \times
\frac{\SO(8)}{\SO(7)} \times \frac{\SO(8)}{\SO(2)\times\SO(6)},
\]
i.e.\ the product of two 7-spheres and the Grassmannian of (oriented) 2-planes in~$\R^8$.
Let us prove that this action is indecomposable and of cohomogeneity one. Let $M_1 := \eS^7 \times \eS^7$ and let $M_2 :=  \Gr_2(\R^8) $. As in the example above, the projection action on~$M_1$ is transitive and obviously the projection action on~$M_2$ is transitive, too. The intersection action on~$M_2$ is conjugate to the $\LG_2$-action on~$\Gr_2(\R^8) $, which is orbit equivalent to the $\SO(7)$-action on~$\Gr_2(\R^8)$, see~\cite{polar}, hence of cohomogeneity one. Now our claim follows from Proposition~\ref{prop:NonSplt}.
\end{enumerate}
\end{examples}

\begin{example}\label{egSpinSeven}
Let us give an example of a cohomogeneity one action as studied in Subsection~\ref{subsec:pp}. The group $\Spin(7)$ acts transitively on $\eS^7$ with (principal) isotropy $\LG_2$. Since $\LG_2$ acts with cohomogeneity one on~$\eS^7$, it follows that an action of~$\Spin(7)$ on~$\eS^7 \times \eS^7$ where both projection actions are given by the $8$-dimensional spin representation is indecomposable and of cohomogeneity one. This action also occurs as an intersection action in Example~\ref{eg:SpinEight}(i).
\end{example}

\begin{examples}\label{eg:telescope}
We give two examples for actions as described in Subsection~\ref{subsec:pnp},
\begin{enumerate}
\item Consider the action of $\U(3) \times \LG_2$ on $\SO(7)$, which is known to be of cohomogeneity one, see~\cite{hyperpolar}. Using the intermediate subgroup $\U(3) \subset \SO(6) \subset \SO(7)$, we may define an action of $\U(3) \times \SO(6) \times \LG_2$ on $\SO(6) \times \LG_2$ given by $(h,\ell,k)\cdot(x,y) = ( h x \ell^{-1}, \ell y k^{-1})$, where $(h,\ell,k) \in \U(3) \times \SO(6) \times \LG_2$ and $(x,y) \in \SO(6) \times \LG_2$. This action is also of cohomogeneity one and indecomposable by Proposition~\ref{prop:NonSplt}.

\item The group $\Spin(9)$ has an isometric action on $\eS^8 \times \eS^{15}$ defined by the standard and the spin representation, cf.~\cite[p.~21]{drk}. The group lift of this action is the action of $\Spin(8) \times \Spin(9) \times \SO(15)$ on $G_1 \times G_2 = \SO(9) \times \SO(16)$. By Lemma~\ref{lm:pnp}, these two actions have the same cohomogeneity as the $\Spin(8)$-action on $\eS^{15}$ which is induced by the direct sum of the two half-spin representations. This $\Spin(8)$-action is cohomogeneity one, this can proved using Table~\ref{t:Trans}. Thus the $\Spin(9)$-action on $\eS^8 \times \eS^{15}$ is of cohomogeneity one and indecomposable by Proposition~\ref{prop:NonSplt}.

\end{enumerate}
\end{examples}


\section{Open Questions}
\label{sec:OpenQ}


An intriguing question that remains open is how one can generalize the results of this article to obtain a classification of cohomogeneity one actions.
The examples in the previous section show that our results will not straightforwardly carry over to the case of indecomposable cohomogeneity one actions. Of course, it will suffice to classify irreducible (in the sense of Definition~\ref{def:Expand}) indecomposable cohomogeneity one actions on symmetric spaces of the compact type, then all other actions can be obtained by expanding factors. Finally, the question arises if it is possible to give a conceptual, i.e.\ classification-free, proof of Theorem~\ref{main:Hermann}.


\bibliographystyle{amsplain}

\end{document}